\documentclass[11pt]{amsart}
\usepackage{Setup}
\usepackage{mathtools}
\usepackage{float}
\mathtoolsset{showonlyrefs}
\renewcommand\Re{\mathrm{Re}}
\renewcommand\mathhlr[1]{#1}
\title[Localized Enhanced Dissipation: A Hypocoercivity Approach]{Localized Enhanced Dissipation: A Hypocoercivity Approach\\ \vspace{0.5cm}
\small  To Eitan, for years of guidance.
}
\date{\today}
\author{Siming He}
\address{Department of Mathematics, University of South Carolina, Columbia, USA}
\email{siming@mailbox.sc.edu}
\thanks{\textbf{Acknowledgment.} Research of SH is supported by NSF grant DMS-2406293.}
\begin{document}
\begin{abstract}
 In this paper, we consider the passive scalar solutions in shear flows with critical points. With a detailed hypocoercivity functional, we develop streamline-wise enhanced dissipation estimates.    
\end{abstract}
\maketitle
\section{Introduction}
In this paper, we consider the passive scalar equation on the two-dimensional domain:
\begin{align}
\pa_t F+U(y)\pa_x F=\nu \de F,\quad F(t=0,x,y)=F_0(x,y). \label{PS}
\end{align}
Here, the solution $F$ represents density, which is transported by the shear flow  $(U(y),0)$ and is subject to diffusion. The dissipation effect in the system is quantified by the small diffusivity $\nu\in(0,1]$. Throughout the paper, the spatial domain is chosen to be $ \Torus^2,\, \Torus\times\rr$ with $\Torus:=\rr/2\pi \mathbbm{Z}$. 

We focus on the \emph{enhanced dissipation phenomenon} associated with equation \eqref{PS}. In the absence of the transport term $U\pa_x f$, the passive scalar equation reduces to the classical heat equation, and the standard energy estimate implies that the solution decays on a time scale $\mathcal{O}(\nu^{-1})$. However, if the transport term is present in the dynamics, the shearing effect of the flow $U$ induces the creation of small spatial scales, which in turn amplifies the dissipation effect. As a consequence of this \emph{small-scale creation (mixing)}, a specific component of the passive scalar solution decays on a time scale that is much faster than the heat dissipation time scale $\mathcal{O}(\nu^{-1})$ (see, e.g., \cite{ElgindiCotiZelatiDelgadino18,FengIyer19}). This phenomenon is called \emph{enhanced dissipation} (\cite{BMV14}) or \emph{relaxation enhancing} (\cite{ConstantinEtAl08}). 

The study of the delicate interaction between the transport and diffusion phenomena dates back to Lord Kelvin (\cite{Kelvin87}), A. Kolmogorov (\cite{Kolmogoroff34}), and L. H\"ormander (\cite{Hormander67}). Their works form the foundation for understanding enhanced dissipation phenomenon in strictly monotone shear flows, e.g., Couette flow $(y,0)$ with $y\in \rr$. To make the enhanced dissipation mathematically precise, we take the $x$-Fourier transform of the equation and obtain the following 
\begin{align}
    \pa_t \wh F_k+Uik \wh F_k=-\nu|k|^2\wh F_k+\nu \pa_y ^2 \wh F_k,\quad \wh F_k(t=0)=\wh F_{0;k},\quad k\in \mathbb Z. \label{PS_k}
\end{align}
The quantitative enhanced dissipation estimate for strictly monotone shear flows reads as follows
\begin{align}
     \|\wh  F_k(t)\|_{L^2_y(\rr)}\leq C \|\wh F_{0;k}\|_{L_y^2(\rr)}e^{-\delta \nu^\f{1}{3}|k|^\f{2}{3} t},\quad k\neq 0.\label{ED_mono}
\end{align}
One observes that as long as the horizontal wave number $k\neq 0$, the solution $\wh F_k$ decays with rate $\mathcal{O}(\nu^\f{1}{3})$, which is much faster than the heat decay rate $\mathcal{O}(\nu)$ in the parameter regime where $0<\nu\ll1$.  Hence, this is the \emph{quantitative description} of \emph{enhanced dissipation phenomenon}. Moreover, this improvement fails for the $k=0$ component, which is governed by the classical heat equation. 

In recent years, extensive research has focused on the enhanced dissipation associated with general shear flows with critical points. These works begin with the remarkable paper \cite{BCZ15}, which employs hypocoercivity techniques from kinetic theory (\cite{villani2009}) to develop enhanced dissipation for a large family of shear flows. To present the result, we restrict ourselves to the shear flow $U(y)=\sin(y)\, (y\in\Torus)$, which has the property that it only has two non-degenerate critical points on the torus $\Torus=\rr/2\pi \mathbb{Z}$. 
It is derived in the work  \cite{BCZ15} that the following hypocoercivity functional decays 
    \begin{align}
        \mathcal{G}[\wh F_k]:=\|\wh F_k\|_{L^2}^2+\al_k\nu^\f12 \|\pa_y \wh F_k\|_{L^2}^2+\beta_k \Re\lan iU' k\wh F_k,\pa_y \wh F_k\ran+\gamma_k\nu^{-\f12}\|U' \wh F_k\|_{L^2}^2. 
    \end{align}
    In fact, for suitably chosen parameters $(\al_k,\beta_k,\gamma_k)$, there exist universal constants  $C\geq 1, \, \delta\in(0,1)$ such that for all $0<\nu\leq 1$, the following estimate holds
    \begin{align}
    \mathcal{G}[\wh  F_k(t)]\leq \mathcal{G}[\wh  F_{0;k}]e^{-\delta \nu^\f12|k|^\f12 t},\quad k\neq 0. \end{align}
The above idea has been refined by various authors \cite{WeiZhang19,LiZhao21,CobleHe23} to obtain a sharper result,
\begin{align}
  \|\wh  F_k(t)\|_{L^2}\leq C \|\wh F_{0;k}\|_{L^2}e^{-\delta \nu^\f12|k|^\f12 t},\quad k\neq 0.\label{ED_intro}
\end{align}
Even though this decay rate $\mathcal{O}(\nu^\f12)$ is much faster than the heat decay rate $\mathcal{O}(\nu)$ when $\nu$ is small, it is much slower than the rate derived for the strictly monotone shear case \eqref{ED_mono}. This discrepancy is rooted in the local geometry of the shear: near critical points, the vanishing of the velocity gradient $U'$ suppresses the generation of small scales via the shearing mechanism, thereby reducing the dissipation enhancement. The sharpness of the estimate \eqref{ED_intro} is confirmed for the solutions concentrating near an $\mathcal{O}(\nu^{\f14})$-neighborhood of the shear profile's critical points, see, e.g., \cite{CotiZelatiDrivas19}.  Subsequently, researchers have developed multiple methodologies to derive enhanced dissipation. These methods include resolvent analysis (\cite{Wei18,He21,LiWeiZhang20,CotiZelatiDolceFengMazzucato,ColomboCotiZelatiWidmayer20}), H\"ormander hypoelliptic method (\cite{AlbrittonBeekie25}), stochastic method (\cite{GardnerLissMattingly24}) and energy method (\cite{ElgindiCotiZelatiDelgadino18}). 

In this paper, we focus on developing the \emph{streamline-wise}/\emph{localized} enhanced dissipation phenomenon. To motivate the problem, we note that even though the work \cite{CotiZelatiDrivas19} confirmed the optimality of the enhanced dissipation rate $\mathcal O(\nu^{\f12})$ for shear flows with nondegenerate critical points, there is a caveat. The quantitative decay \eqref{ED_intro} is only known to be sharp for solutions concentrated near the critical layer surrounding the critical point $y_{\rm crit}$ (Figure \ref{fig:Crit_Layer}). In reality, the initial data can be distributed anywhere on the torus. If the support of the initial data is far away from the critical point $y_{\rm crit}$ of $U$, can we obtain better results? In Figure \ref{fig:Crit_Layer}, we present a detailed decomposition of the spatial domain based on the shear profile. One expects the following description. In the critical layer, the fluctuation decays with enhanced dissipation rate $\mathcal{O}(\nu^\f12)$, whereas in the monotone region ($|U'|\sim 1$), this decay rate improves to $\mathcal{O}(\nu^\f{1}{3})$. This transition motivates the concept of \emph{streamline-wise/local} enhanced dissipation. One expects that the enhanced dissipation rate should transit smoothly from $\mathcal{O}(\nu^\f12)$ (critical layer)  to $\mathcal{O}(\nu^\f{1}{3})$ (monotone region) within the transition layer. Recently, there are increasing interest in developing these estimates (\cite{GardnerLissMattingly24,AlbrittonBeekie25}). In \cite{GardnerLissMattingly24}, the authors applied delicate stochastic analysis to develop the enhanced dissipation. In \cite{AlbrittonBeekie25}, the authors applied detailed resolvent analysis. 
\begin{figure}[H]
    \centering
\includegraphics[width=7.5cm]{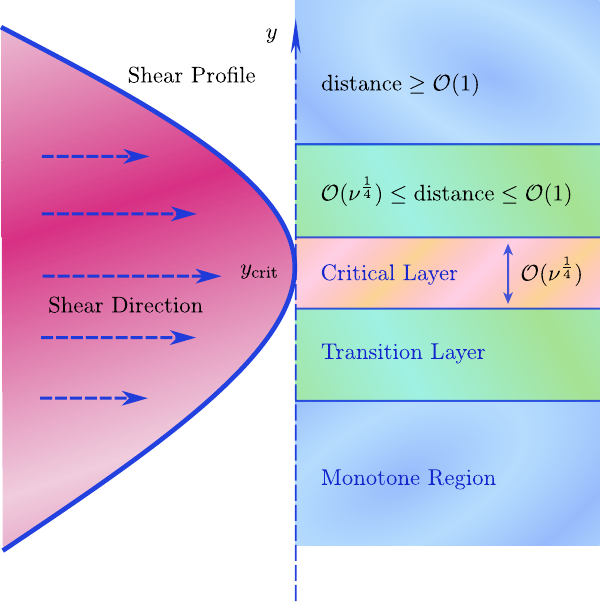}
    \caption{Critical Layer}
    \label{fig:Crit_Layer}
\end{figure}%conjectured ($k=1$) 

To analyze the system \eqref{PS_k} and establish the streamline-wise enhanced dissipation, we consider the following special Fourier coefficient
\begin{align} 
f(t,y):=\wh F_1(t,y)e^{\nu |k|^2t}. 
\end{align}
The quantity $f(t,y)$ solves the following  hypoelliptic equation
\begin{align}
   \pa_t f+iU f=\nu \pa_y^2f,\quad f(t=0,y)=f_0(y)=\wh F_{0;1}(y).  \label{EQ:hypo_PS}
\end{align} It can be checked that the other general cases with $k\neq 0$ can be reduced to \eqref{EQ:hypo_PS}. For the detailed argument, see, e.g. \cite{He21}.

To define a physically relevant hypocoercivity functional, we need a spatial-temporal weight that captures the streamline-wise dissipation. Define the modified shear strength \begin{align}  &B(y):=\max\{|U'(y)|,\nu^{1/4}\}. \label{B}\end{align} This quantity is introduced to measure the effective shearing of the flow. We note that $B\equiv\nu^{\f14}$ in the critical layer. The reason is that even though the shear gradient $U'$ is vanishingly small near the critical point, the solution can still diffuse in the critical layer, and the experience of a nontrivial shearing effect can be approximated by a $\nu$-dependent constant. With the modified shear strength $B$, we define the spatial-temporal weight $W$, \begin{align} 
\label{W}
&W(t,y):=\exp\Bigl\{\varsigma \nu^\f{1}{3}B^{\f23}(y)\max\bigl\{\nu^{-\f{1}{3}}B^{-\f23}(y),\min\{t,\nu^{-\f12}\}\bigr\}\Bigr\},\quad \varsigma\in(0,1). 
\end{align}In specific space-time regimes, the weight $W$ provides a continuous interpolation between the enhanced dissipation effects in the critical layer and in the monotone region. 
In Figure \ref{fig:1}, we can find a sketch of the exponential part of the spatial-temporal weight $W$. The analytical properties of the weight $W$ are summarized in Lemma \ref{lem:W_est}. 
\begin{figure}[H]
    \centering
    \includegraphics[width=0.5\linewidth]{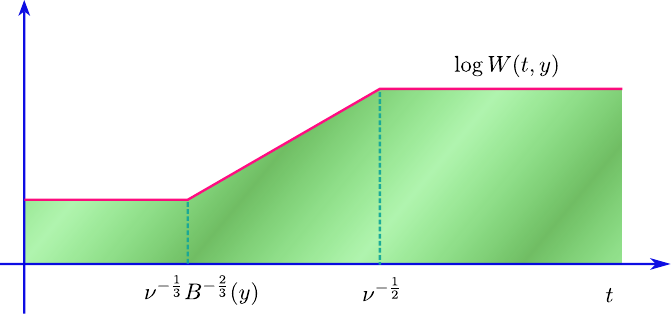}
    \caption{The weight $\log W$.}
    \label{fig:1}
\end{figure}
We consider the following hypocoercivity functional
    \begin{align}\label{Hypo_F}\begin{split}
\Phi[f]:=&\|fW\|_{L^2}^2+ {\frac{1}{4}\beta^\f12} \nu^{\f23}\|\sqrt{\varphi} B^{-\f{1}{3}}\pa_y f W\|_{L^2}^2+\beta\nu^{\f{1}{3}}\Re\lan \varphi^2 B^{-\f43} iU' f,\pa_y fW^2\ran\\ &+ {4\beta^\f32}\|\varphi^\f32 U'f B^{-1}W\|_{L^2}^2. 
    \end{split}
    \end{align} 
Here, the following weight is introduced to ensure that $\Phi[f_0]\approx\|f_0\|_{L^2}^2:$  \begin{align} \label{varphi}
&\varphi(t,y):=\min\{1,\nu^\f{1}{3}\max\{|U'(y)|,\nu^\f14\}^{\f23}t\}=\min\{1,\nu^\f{1}{3} B^{\f23}(y)t\}.
\end{align}
    Throughout the paper, we assume that $\beta(U)\in(0,1]$ with a specific value chosen in the proof. The classical $\al,\gamma$-parameters in the hypocoercivity are chosen as $\al :=\frac{1}{4}\beta^\f12, \; \gamma: = 4\beta^\f32 $. Moreover, we use the notation $G'(t,y):=\pa_y G(t,y)$.  %In this note, the main goal is to understand the behavior of the passive scalar equation prior to the enhanced dissipation time scale $\mathcal{O}(\nu^{-\f12})$. 
    We present the following main theorem. %estimate %\begin{align} %   \Phi[f(t)]\leq \Phi[f_0]=\|f_0\|_{L^2}^2,\quad \forall t\in[0, \nu^{-\f12}].     %\end{align}
\begin{thm}\label{thm:ED}
Consider the smooth solution $f$ to the equation \eqref{EQ:hypo_PS}. Further assume that the shear flow profile $U$ only has finitely many nondegenerate critical points and the diffusion coefficient $\nu\in(0,1]$. Then the  functional $\Phi$ \eqref{Hypo_F} has the following estimate
\begin{align}
    \|f(t)W(t)\|_{L^2}^2\leq \Phi[f(t)]\leq e^3\|f_0\|_{L^2}^2 e^{-\delta \nu^\f12 t},\quad \forall t\geq0.\label{SL_ED}
\end{align}
Here, $\delta\in(0,1)$ is a constant that depends only on $U.$
    \end{thm}
In the main theorem, we observe that the global enhanced dissipation rate is still $\mathcal{O}(\nu^{\f12})$, which is consistent with the result in \cite{CotiZelatiDrivas19}. However, thanks to the presence of the spatial-temporal  weight $W$ \eqref{W}, the solution decays with an improved rate $\mathcal{O}(\nu^{\f{1}{3}}\max\{|U'|,\nu^{\f14}\}^{\f23})$ along streamlines which are beyond the critical layer (Figure \ref{fig:Crit_Layer}). In particular, the estimate recovers the quantitative bound \eqref{ED_mono} in the monotone region. We conjecture that the estimate is sharp in the transition layer. 

Finally, we note that the enhanced dissipation has found applications in various areas, ranging from fluid mechanics (\cite{BMV14,BVW16,BGM15I,BGM15II,BGM15III,ChenLiWeiZhang18,BedrossianHe19,BedrossianHeIyerWang23,CotiZelatiElgindiWidmayer20,WeiZhang19,BeekieHe26}) to kinetic theory (\cite{Bedrossian17,ChaturvediLukNguyen23,BedrossianCotiZelatiDolce22}) and mathematical biology (\cite{BedrossianHe16,KiselevXu15,IyerXuZlatos, HeTadmor172,HeTadmorZlatos,He, He22, HeKiselev21,GongHeKiselev21,AlbrittonOhm22,CotiZelatiDietertGerardVaret22,CotiZelatiDietertVaret24,HuKiselevYao23,GuHe25,He25_I,He25_II}). In many of these works (\cite{BedrossianHe16,BedrossianCotiZelatiDolce22,LiZhao21}), the application of the hypocoercivity functional enables one to treat challenging nonlinear problems. Consequently, we believe that the hypocoercity framework presented in this paper will have future applications.
    \section{Proof of the Main Theorem}
\begin{lemma}\label{lem:eqv}
    The following equivalence relation concerning the functional $\Phi$%\mathcal{G}$ \eqref{hyp_fnctnl} holds
    \begin{align}\label{equiv_ndg}
       \Phi[f]\approx \|fW\|_{L^2}^2+ \beta^\f12  \nu^{\f23}\|\sqrt{\varphi} B^{-\f{1}{3}}\pa_y f W\|_{L^2}^2+\beta^\f32\|\varphi^\f32 U'f B^{-1}W\|_{L^2}^2.
       %\mathcal{G}[f] \leq d^2\|f\|_2^2 + \frac{3}{2}\lf(\al\phi\ep^{1/2}\|\pa_{y}f\|_2^2 + \gamma\phi^3\ep^{-1/2}\|\cos(y)f\|_2^2\rg).
    \end{align}
\end{lemma}
\begin{proof}
This is a consequence of the observation
\begin{align}
    \beta\nu^{\f{1}{3}}|\Re\lan \varphi^2 B^{-\f43} iU' f,\pa_y fW^2\ran|\leq {\frac{1}{8}\beta^\f12} \nu^{\f23}\|\sqrt{\varphi} B^{-\f{1}{3}}\pa_y f W\|_{L^2}^2+{2\beta^\f32}\|\varphi^\f32 U'f B^{-1}W\|_{L^2}^2.
\end{align}
    % We recall the definition of $\mathcal{G}$ \eqref{hyp_fnctnl}, and estimate $\mathcal{G}[f]$ using H\"older inequality and Young's inequality,
    % \begin{align}
    %     \mathcal{G}[f] \leq& d^2\|f\|_2^2 + \al\phi\ep^{1/2}\|\pa_{y}f\|_2^2 + \beta\phi^2\|\cos(y)f\|_2\|\pa_{y}f\|_2 + \gamma\phi^3\ep^{-1/2}\|\cos(y)f\|_2^2\\
    %     \leq& d^2\|f\|_2^2 + \frac{3\al}{2}\phi\ep^{1/2}\|\pa_{y}f\|_2^2 + \lf(\gamma + \frac{\beta^2}{2\al}\rg)\phi^3\ep^{-1/2}\|\cos(y)f\|_2^2.
    % \end{align}
    % Similarly, we have the lower bound,
    % \begin{align}
    %     \mathcal{G}[f] \geq d^2 \|f\|_2^2 + \frac{\al}{2}\phi\ep^{1/2}\|\pa_{y}f\|_2^2 + \lf(\gamma - \frac{\beta^2}{2\al}\rg)\phi^3\ep^{-1/2}\|\cos(y)f\|_2^2.
    % \end{align}
    % Since \eqref{ndeg_bnd_req} implies that $\frac{\beta^2}{2\al} \leq \frac{\gamma}{2}$, we obtain that
    % \begin{align}
    %    d^2 \|f\|_2^2 + \frac{1}{2}\al\phi\ep^{1/2}\|\pa_{y}f\|_2^2 + \frac{1}{2}\gamma\phi^3\ep^{-1/2}\|\cos(y)f\|_2^2 \leq \mathcal{G}[f] \leq d^2\|f\|_2^2 + \frac{3}{2}\al\phi\ep^{1/2}\|\pa_{y}f\|_2^2 + \frac{3}{2}\gamma\phi^3\ep^{-1/2}\|\cos(y)f\|_2^2.
    % \end{align}
    This concludes the proof of the lemma.
\end{proof}
By taking the time derivative of the hypocoercivity functional, \eqref{Hypo_F}, we end up with the following decomposition:%\mathcal{G}
\begin{align}\begin{split}
    \frac{d}{dt}\Phi[f(t)] =& \frac{d}{dt}\|fW\|_{L^2}^2 + \frac{d}{dt}\lf(\frac{1}{4}\beta^\f12 \nu^{\f23}\|\sqrt{\varphi} B^{-\f{1}{3}}\pa_y f W\|_{L^2}^2\rg)\\
    &+ \frac{d}{dt}\lf(\beta\nu^{\f{1}{3}}\Re\lan \varphi^2 B^{-\f43} iU' f,\pa_y fW^2\ran\rg) \\
    &+ \frac{d}{dt}\lf(4\beta^\f32\|\varphi^\f32 U'f B^{-1}W\|_{L^2}^2\rg)\\
    =:& T_{L^2} + T_\al+ T_\beta + T_\gamma.
\end{split}\label{ndeg_T_albe_term}
\end{align}
\begin{lemma}\label{lem:L2}
Consider the $T_{L^2}$-term in \eqref{ndeg_T_albe_term}. The following estimate holds
\begin{align}
   T_{L^2}\leq -\f32\nu\|\pa_y fW\|_{L^2}^2+(2\varsigma+ 5\varsigma^2\|U''\|_{L^\infty}^2\min\{\nu t^2,1\})\nu^\f{1}{3}\|f \varphi B^{\f{1}{3}}W\|_{L^2}^2.\label{T_L2}
\end{align}
\end{lemma}
\begin{proof}
    
We invoke the $W$-estimates \eqref{W_est_t} and \eqref{W_est_y}, and the crucial fact that $\varphi=\min\{\nu^\f{1}{3}B^{\f23}t,1\}=1$ for $t\geq\nu^{-\f{1}{3}}B^{-\f23}$ to obtain that 
\begin{align}
    \begin{split}T_{L^2}\leq&-2\nu\|\pa_y fW\|_{L^2}^2+2\varsigma\nu^\f{1}{3}\int|f|^2 B^{\f23}W^2\mathbbm{1}_{t\in[\nu^{-\f{1}{3}}B^{-\f23}, \nu^{-\f12}]}dy+2\nu\Re \int \overline{f}\pa_y f\pa_y(W^2)dy\\
    \leq&-2\nu\|\pa_y fW\|_{L^2}^2+2\varsigma\nu^\f{1}{3}\|f \varphi B^{\f{1}{3}}W\|_{L^2}^2+\frac{8}{3}\nu\varsigma\nu^{\f{1}{3}}\min\{t,\nu^{-\f12}\}\int \mathbbm{1}_{t\geq\nu^{-\f{1}{3}}B^{-\f23}}|\pa_y f||f| |B'|B^{-\f{1}{3}} W^2dy\\
    \leq &-\f32\nu\|\pa_y fW\|_{L^2}^2+2\varsigma\nu^\f{1}{3}\|f \varphi B^{\f{1}{3}}W\|_{L^2}^2+5\varsigma^2\nu^\f43\|U''\|_{L^\infty}^2\min\{t^2,\nu^{-1}\}\|B^{-\f43}\|_{L^\infty}\nu^\f{1}{3}\|f \varphi B^{\f{1}{3}}W\|_{L^2}^2\\
    \leq &-\f32\nu\|\pa_y fW\|_{L^2}^2+(2\varsigma+ 5\varsigma^2\|U''\|_{L^\infty}^2\min\{\nu t^2,1\})\nu^\f{1}{3}\|f \varphi B^{\f{1}{3}}W\|_{L^2}^2. \end{split}
\end{align}

\end{proof}
The estimates for the $T_\al,\ T_\beta$, and $T_\gamma$ terms are tricky, and we collect them in the following technical lemmas whose proofs will be postponed to the end of this section.
\begin{lemma}[$\al$-estimate]\label{lem:nondegenerate al}
	The following estimate holds:
	\begin{align}\label{nondegenerate al estimate}
            T_\al \leq &\lf(\frac{1}{2}+\frac{3}{4}\beta^\f12+2\beta^{\f12}\|U''\|_{L^\infty}^2+4\varsigma^2\beta^\f12\|U''\|_{L^\infty}^2\min\{1,\nu t^2\}\rg)\nu\|\pa_{y}f W\|_2^2 \\
            &- \f{5}{12}\beta^\f12\nu^{\f53} \|\pa_y^2f \sqrt{\varphi} B^{-\f{1}{3}}W\|_2^2 +\frac{1}{8}\beta\nu^{\f{1}{3}}\|U' f \varphi B^{-\f23}W\|_{L^2}^2.
            \end{align} 
\end{lemma}
\begin{lemma}[$\beta$-estimate]\label{lem:nondegenerate beta}
	The following estimate holds
	\begin{align}\label{nondegenerate beta estimate}\begin{split}
         T_\beta \leq& -\beta\lf(\f12-32\sqrt{\beta}-64\sqrt{\beta}\|U''\|_{L^\infty}^2\rg)\nu^\f{1}{3}\|U' f \varphi B^{-\f23}W\|_{L^2}^2+  8\beta \nu \|\pa_{y}f W\|_{L^2}^2\\ &+\mathhlr{\frac{5}{12}}\beta^{\f12}\nu^{\f53} \|\pa_{y}^2f\sqrt{\varphi} B^{-\f{1}{3}}W\|_2^2 +  5\beta^\f32\nu\|U'\pa_y f\varphi^\f32B^{-1}W\|_2^2\\
        &+12\beta^{\f32}\nu^{\f{1}{3}}\|U''\|_{L^\infty}^2\|f\varphi B^{\f{1}{3}}W\|_{L^2}^2.\end{split}
	\end{align}
\end{lemma}
\begin{lemma}[$\gamma$-estimate]\label{lem:nondegenerate gamma}
	The following estimate holds,
	\begin{align}\label{nondegenerate gamma estimate}
        \begin{split}
            T_{\gamma}
        \leq& -8\beta^\f32\nu\|U'\pa_{y}f\varphi^\f32 B^{-1}W\|_{L^2}^2+(48+8\|U''\|_{L^\infty}^2)\beta\nu\|\pa_y f W\|_{L^2}^2\\
        &+ \lf[28+\lf(8+\f{16}{9}\varsigma^2\min\{1, \nu t^2\}\rg)\|U''\|_{L^\infty}^2\rg]\beta^\f32\nu^\f{1}{3}\|U' f\varphi B^{-\f23}W\|_{L^2}^2.  
        \end{split}
	\end{align}
\end{lemma}
These estimates allow us to prove Theorem \ref{thm:ED}.
\begin{proof}[Proof of Theorem~\ref{thm:ED}]  
    Combining the estimates \eqref{T_L2}, \eqref{nondegenerate al estimate}, \eqref{nondegenerate beta estimate}, \eqref{nondegenerate gamma estimate} and setting \[\varsigma=\beta^{\f32},\]  we obtain that for $\beta\leq1$, 
    \begin{align}\begin{split}
    \frac{d}{dt}\Phi[f(t)] 
         \leq& -\f32\nu\|\pa_y fW\|_{L^2}^2+(2\varsigma+ 5\varsigma^2\|U''\|_{L^\infty}^2)\nu^\f{1}{3}\|f \varphi B^{\f{1}{3}}W\|_{L^2}^2\\
         %%%%%%%%%%%%%%%%%%%%%%%%%
         &+\lf(\frac{1}{2}+\frac{3}{4}\beta^\f12+6\beta^{\f12}\|U''\|_{L^\infty}^2\rg)\nu\|\pa_{y}f W\|_2^2 \\
            &- \f{5}{12}\beta^\f12\nu^{\f53} \|\pa_y^2f \sqrt{\varphi} B^{-\f{1}{3}}W\|_2^2 +\frac{1}{8}\beta\nu^\f{1}{3}\|U' f \varphi B^{-\f23}W\|_{L^2}^2\\
            %%%%%%%
         &-\beta\lf(\f12-32\sqrt{\beta}-64\sqrt{\beta}\|U''\|_{L^\infty}^2\rg)\nu^\f{1}{3}\|U' f \varphi B^{-\f23}W\|_{L^2}^2+  8\beta \nu \|\pa_{y}f W\|_{L^2}^2\\ &+\frac{5}{12}\beta^{\f12}\nu^{\f53} \|\pa_{y}^2f\sqrt{\varphi} B^{-\f{1}{3}}W\|_2^2 +  5\beta^\f32\nu\|U'\pa_y f\varphi^\f32B^{-1}W\|_2^2\\
        &+12\beta^{\f32}\nu^{\f{1}{3}}\|U''\|_{L^\infty}^2\|f\varphi B^{\f{1}{3}}W\|_{L^2}^2\\
    %%%%%%%%
&-8\beta^\f32\nu\|U'\pa_{y}f\varphi^\f32 B^{-1}W\|_{L^2}^2+(48+8\|U''\|_{L^\infty}^2)\beta\nu\|\pa_y f W\|_{L^2}^2\\
        &+ \lf[28+10\|U''\|_{L^\infty}^2\rg]\beta^\f32\nu^\f{1}{3}\|U' f\varphi B^{-\f23}W\|_{L^2}^2\\ 
        %%%%%%%%%%%%%%%%%%%
        \leq&-\left(1-\frac{3}{4}\beta^{\f12}-14\beta^{\f12}\|U''\|_{L^\infty}^2-56\beta\right)\nu\|\pa_y fW\|_{L^2}^2\\
        &+\lf(10\beta^{\f32}+17\beta^{\f32}\|U''\|_{L^\infty}^2\rg)\nu^\f{1}{3}\|f \varphi B^{\f{1}{3}}W\|_{L^2}^2\\
        &-\beta\lf(\f38-60\sqrt{\beta}-74\sqrt{\beta}\|U''\|_{L^\infty}^2\rg)\nu^\f{1}{3}\|U' f \varphi B^{-\f23}W\|_{L^2}^2\\
        &-3\beta^\f32\nu\|U'\pa_{y}f\varphi^\f32 B^{-1}W\|_{L^2}^2.
        \end{split}  
    \end{align}
    %We choose $\al$, $\gamma$ in terms of $\beta$ as follows 
    %\begin{align}\label{chc_al_ga}
    %    \al = \frac{\beta^{1/2}}{4}, \quad\quad \gamma = 4\beta^{3/2} .
    %\end{align}
Since we assume that $\nu\leq 1$, we can invoke the spectral inequality \eqref{spectral} to obtain,
      \begin{align*}
    \frac{d}{dt}\Phi[f(t)]\leq & -\left(1-\frac{3}{4}\beta^{\f12}-14\beta^{\f12}\|U''\|_{L^\infty}^2-56\beta\mathhlr{-20\beta^\f32-17\beta^{\f32}\|U''\|_{L^\infty}^2}\right)\nu\|\pa_y fW\|_{L^2}^2\\
        &-10\beta^\f32\nu^\f{1}{3}\|f \varphi B^{\f{1}{3}}W\|_{L^2}^2\\
        &-\beta\lf(\f38-64\sqrt{\beta}-74\sqrt{\beta}\|U''\|_{L^\infty}^2-20\mf C\beta^\f12-17\beta^{\f12}\mf C\|U''\|_{L^\infty}^2\rg)\nu^\f{1}{3}\|U' f \varphi B^{-\f23}W\|_{L^2}^2\\
        &-3\beta^\f32\nu\|U'\pa_{y}f\varphi^\f32 B^{-1}W\|_{L^2}^2.%MATH
    \end{align*}
    Hence, we can choose
    \begin{align}
    \beta=\beta(\mathfrak{C},\|U''\|_{L^\infty})<1,\label{chc_beta}
    \end{align}
small enough such that %mathcal{G}
    \begin{align}    \frac{d}{dt}\Phi[f(t)]\leq&- \frac{1}{2}\nu\|\pa_y fW\|_{L^2}^2-10\beta^\f32\nu^\f{1}{3}\|f \varphi B^{\f{1}{3}}W\|_{L^2}^2 -\beta\f12\nu^\f{1}{3}\|U' f \varphi B^{-\f23}W\|_{L^2}^2
    \\
    \leq&- \frac{1}{2}\nu^{\frac{1}{2}}\nu^{\f23}\mathbbm{1}_{t\geq\nu^{-\f12}}\|\pa_y f  B^{-\f{1}{3}}W\|_{L^2}^2-10\beta^\f32\nu^\f12\mathbbm{1}_{t\geq\nu^{-\f12}}\|f  W\|_{L^2}^2 -\beta\f12\nu^\f{1}{2}\mathbbm{1}_{t\geq\nu^{-\f12}}\|U' f B^{-1}W\|_{L^2}^2\\
    \leq &-\delta\nu^\f12\mathbbm{1}_{t\geq\nu^{-\f12}}\Phi[f(t)].\label{dfn_del}
    \end{align} 
   % \myb{Need to drop $\mathhlr{\varphi}!$} 
   Here in the last line, we have invoked the spectral inequality \eqref{spectral} and the comparison Lemma \ref{lem:eqv}. Finally, we observe that the parameter $\delta\in(0,1]$ depends only on the parameter $\mathfrak{C},\|U''\|_{L^\infty}$. So for $t\leq \nu^{-\frac{1}{2}}$, we have that 
   \begin{align}
       \Phi[f(t)]\leq \Phi[f_0]=e^{\delta}\|f_0W(t=0)\|_{L^2}^2 e^{-\delta \nu^{\f12}t}\leq e^3\|f_0\|_{L^2}^2 e^{-\delta \nu^{\f12}t}.
   \end{align}
   For $t\geq \nu^{-\f12}$, 
   \begin{align}
       \Phi[f(t)]\leq \Phi[f(\nu^{-\f12})]e^{-\delta \nu^{\f12}(t-\nu^{-\f12})}\leq e^\delta\Phi[f_0]e^{-\delta \nu^{\f12}t}\leq e^3\|f_0\|_{L^2}^2 e^{-\delta \nu^{\f12}t}.
   \end{align}
   As a consequence, we have proven \eqref{SL_ED}, and hence Theorem \ref{thm:ED}.
\end{proof}
We conclude the section by providing the details of the proof of Lemma \ref{lem:nondegenerate al}, \ref{lem:nondegenerate beta}, and \ref{lem:nondegenerate gamma}. 
\begin{proof}[Proof of Lemma~\ref{lem:nondegenerate al}] We recall the definition of $T_\al$~\eqref{ndeg_T_albe_term} and define \begin{align}\label{chc_al_ga}
        \al = \frac{\beta^{1/2}}{4}.%, \quad\quad \gamma = 4\beta^{3/2} 
    \end{align} Invoking the equation~\eqref{EQ:hypo_PS} and integration by parts yields that
    \begin{align}\begin{split}
         T_\al
         \leq&  \al\nu^{\f23}\int \mathbbm{1}_{t\leq \nu^{-\f{1}{3}}B^{-\f23}}\nu^{\f{1}{3}} B^\f23\ |\pa_{y}f|^2\ B^{-\f23}W^2dy +\al\phi\nu^\f23\int|\pa_y f|^2\ B^{-\f23} \ 2\varsigma\nu^\f{1}{3} B^\f23 W^2dy \\
         &+2\al\nu^{\f23}\Re\int \pa_{y}(\nu\pa_{y}^2f - iU f) \overline{\pa_{y} f}\ \varphi B^{-\f23} W^2 dy\\
         \leq& \al(1+2\varsigma)\nu\|\pa_{y}f W\|_2^2 - 2\al\nu^{\f53} \|\pa_y^2f \sqrt{\varphi} B^{-\f{1}{3}}W\|_2^2 \\
         &- 2\al\nu^{\f53} \int \pa_y^2f\overline{\pa_y f}\ \pa_y{\varphi}\ B^{-\f23}\ W^2 dy- 2\al\nu^{\f53} \int \pa_y^2f\overline{\pa_y f}\ {\varphi}\ \pa_y(B^{-\f23}) \ W^2dy\\
         &- 2\al\nu^{\f53} \int \pa_y^2f\overline{\pa_y f} \ {\varphi}\ B^{-\f23}\ \pa_y(W^2)dy- 2\al\nu^{\f23}\Re\int iU'f\overline{\pa_{y}f}\ \varphi B^{-\f23}W^2dy\\
         =:&\sum_{j=1}^6 T_{\al;j}.\end{split}\label{T_al_j}
       %  \leq&\al\nu\mathbbm{1}_{t\leq \ep^{-\f12}}\|\pa_{y}f\|_2^2 - 2\al\phi\ep^{1/2} \nu\|\pa_y\pa_{y}f\|_2^2 + \frac{\beta\phi^2}{B}|k_m|\|\cos(y)f_\bk\|_{L^2}^2+\frac{B\al^2}{\beta}\nu \underbrace{|k_m|^{-2}\Bigl|\sum_{j=1}^d k_j\Bigr|^2}_{\leq d^2}\|\pa_{y}f\|_{L^2}^2.
    \end{align}
    
    Now we estimate each term in the expression \eqref{T_al_j}. To begin with, we consider $T_3$. Recalling the definition of $B\ \eqref{B},\ \varphi \ \eqref{varphi}$, and observe that $|B'(y)|\leq |U''(y)|\mathbbm{1}_{\{|U'|\geq \nu^{\f14}\}}(y)$ (a.e.), we obtain  
    \begin{align}\begin{split}
    |T_{\al;3}|\leq &\frac{4}{3}\al\nu^\f53\int |\pa_y^2 f||\pa_y f|\mathbbm{1}_{t\leq \nu^{-\f{1}{3}}B^{-\f23}}\nu^{\f{1}{3}}B^{-\f{1}{3}}|B'| t \ B^{-\f23}W^2dy\\
    \leq&\frac{4}{3}\al\nu^\f53\int |\pa_y^2 f||\pa_y f| |B'|\ \varphi  B^{-\f23- 1}W^2dy\\
    \leq &\frac{1}{9}\al\nu^{\f53}\|\pa_y^2 f \sqrt{\varphi}B^{-\f{1}{3}} W\|_{L^2}^2+4\al\|U''\|^2_{L^\infty}\nu^{\f53}\int |\pa_yf|^2 \varphi B^{-8/3} W^2dy\\
    \leq &\frac{1}{9}\al\nu^{\f53}\|\pa_y^2 f \sqrt{\varphi}B^{-\f{1}{3}} W\|_{L^2}^2+4\al\nu\|U''\|^2_{L^\infty}\|\pa_yf W\|_{L^2}^2. \end{split}\label{T_al_3}
    \end{align}
    Here, in the last line, we have used $B\geq\nu^\f14,\, \varphi\leq1$. 

    Next, we estimate the $T_{\al;4}$-term in the expansion \eqref{T_al_j}. We have that 
    \begin{align}
        \begin{split}
            |T_{\al;4}|\leq &\f43\al \nu^\f53\int|\pa_y^2f||\pa_yf| B^{-1}|B'|\ \varphi B^{-\f23} W^2dy\\
            \leq &\frac{1}{9}\al\nu^\f53\|\pa_y^2 f \sqrt{\varphi}B^{-\f{1}{3}}W\|_{L^2}^2+4\al\nu^{\f53}\|U''\|_{L^\infty}^2\int|\pa_yf|^2B^{-\f83}\varphi W^2dy\\
            \leq &\frac{1}{9}\al\nu^\f53\|\pa_y^2 f \sqrt{\varphi}B^{-\f{1}{3}}W\|_{L^2}^2+4\al\nu\|U''\|_{L^\infty}^2\|\pa_yfW\|_{L^2}^2.
        \end{split}\label{T_al_4}
    \end{align}
    
     We estimate the $T_{\al;5}$-term as follows
    \begin{align}\begin{split}
         |T_{\al;5}|\leq &\f83\al \nu^\f53\int|\pa_y^2f||\pa_yf| \varsigma \nu^\f{1}{3} B^{-\f{1}{3}}|B'| \mathhlr{\min\{t,\nu^{-\f12} \}} \ \varphi B^{-\f23} W^2dy\\
            \leq &\frac{1}{9}\al\nu^\f53\|\pa_y^2 f \sqrt{\varphi}B^{-\f{1}{3}}W\|_{L^2}^2+16\varsigma^2\al\nu^{\f73}\mathhlr{\min\{t^2,\nu ^{-1}\}}\|U''\|_{L^\infty}^2\int|\pa_yf|^2B^{-\f43}\varphi W^2dy\\
            \leq &\frac{1}{9}\al\nu^\f53\|\pa_y^2 f \sqrt{\varphi}B^{-\f{1}{3}}W\|_{L^2}^2+16\varsigma^2\al\nu\|U''\|_{L^\infty}^2\mathhlr{\min\{1,\nu t^2\}}\|\pa_yfW\|_{L^2}^2.
        \end{split}\label{T_al_5}%MATH\mathbbm{1}_{t\leq \nu^{-\f12}} \mathbbm{1}_{t\leq \nu^{-\f12}}
    \end{align}

Finally, we estimate the $T_{\al;6}$-term as follows
\begin{align}
\begin{split}
     |T_{\al;6}|\leq \frac{1}{2}\nu\|\pa_y f W\|_{L^2}^2+2\al^2\nu^\f{1}{3}\|U'f\varphi B^{-\f23}W\|_{L^2}^2.
\end{split}    \label{T_al_6}
\end{align}

  These estimates, when combined with the decomposition \eqref{T_al_j}, yields the result \eqref{nondegenerate al estimate}.
\end{proof}
\begin{proof}[Proof of Lemma~\ref{lem:nondegenerate beta}]
    The estimate of the $T_\beta$ term in \eqref{ndeg_T_albe_term} is technical. Hence, we further decompose it into four terms and estimate them one by one:
    \begin{align}\label{T_beta123}\begin{split}
        T_\beta =& 2\beta\nu^\f{1}{3}\Re\langle iU'f,\pa_{y}f\ \varphi \pa_t\varphi\ B^{-\f43}W^2\rangle +2\beta\nu^\f{1}{3}\Re\langle iU'f,\pa_{y}f \varphi^2 B^{-\f43}\ W\pa_tW\rangle\\
        &+ \beta\nu^\f{1}{3}\Re\int iU'\pa_tf\overline{\pa_{y}f} \ \varphi^2B^{-\f43}W^2dy + \beta\nu^\f{1}{3}\Re\int iU'f\overline{\pa_{y t}f}\ \varphi^2B^{-\f43}W^2dy\\
        =:& \sum_{j=1}^4T_{\beta;j}.\end{split}
    \end{align}
    To begin with, we have the following bound for the $T_{\beta;1}$ using H\"older and Young's inequalities
    \begin{align*}
        |T_{\beta;1} |\leq& 2\beta\nu^\f{1}{3}\int|U'f||\pa_{y}f| \mathbbm{1}_{t\leq \nu^{-\f{1}{3}}B^{-\f23}}\nu^\f{1}{3}B^{\f23} \ \varphi B^{-\f43} W^2 dy
        \leq \frac{\beta}{4}\nu^{\f{1}{3}}\lf\|U'f\varphi B^{-\f23}W\rg\|_{L^2}^2 +  4\beta \nu \|\pa_{y}f W\|_{L^2}^2.
    \end{align*}
    
Next we compute the term $T_{\beta;2}$:
\begin{align*}
        |T_{\beta;2}|\leq &2\beta\nu^\f{1}{3}\int |U' f| |\pa_y f|\ \varsigma\nu^{\f{1}{3}}B^\f23 \ \varphi^2 B^{-\f43} W^2dy 
        \leq  \frac{\beta}{4}\nu^{\f{1}{3}}\varsigma\lf\|U'f\varphi B^{-\f23}W\rg\|_{L^2}^2 +  4\varsigma\beta \nu \|\pa_{y}f W\|_{L^2}^2.
\end{align*}
Now we move on to the third term, 
    \begin{align*}
        T_{\beta;3} =& \beta\nu^\f{1}{3}\Re\int iU'\lf(\nu\pa_{y}^2f-iU f\rg)\overline{\pa_{y}f} \varphi^2 B^{-\f43}W^2dy\\
        =& \beta\nu^\f{4}{3}\Re\int iU'\pa_y^2 f \overline{\pa_{y}f} \varphi^2 B^{-\f43}W^2dy+ \beta\nu^\f{1}{3}\Re\int U'Uf\overline{\pa_{y}f} \varphi^2 B^{-\f43}W^2dy\\
        \leq& \frac{1}{4}\beta^{\f12}\nu^{\f53} \|\pa_{y}^2f\sqrt{\varphi} B^{-\f{1}{3}}W\|_2^2 +  \beta^\f32\nu\|U'\pa_y f\varphi^\f32B^{-1}W\|_2^2\\
        &\mathhlr{+ \beta\nu^\f{1}{3}\Re\int U'Uf\overline{\pa_{y}f} \varphi^2 B^{-\f43}W^2dy}.
    \end{align*}
    Finally we estimate the term $T_{\beta;4}$ in~\eqref{T_beta123}
    \begin{align}\label{T_beta4}\begin{split}
        T_{\beta;4} =& \beta\nu^\f{1}{3}\Re\int iU' f\overline{\lf(\nu\pa_{y}^3f - iU'f - iU\pa_{y}f\rg)}\varphi^2 B^{-\f43} W^2dy\\
        =& -\beta\nu^\f{4}{3}\Re\int i\lf(U'' f +U'\pa_y f\rg)\overline{\pa_{y}^2 f}\varphi^2 B^{-\f43}W^2 dy \\
        &-\beta\nu^\f{4}{3}\Re\int iU' f\overline{\pa_{y}^2 f}\pa_y(\varphi^2 B^{-\f43}W^2 )dy\\
        &- \beta\nu^\f{1}{3}\int|U'|^2|f|^2\varphi^2 B^{-\f43} W^2 dy -\beta\nu^\f{1}{3}\Re\int U 'Uf\overline{\pa_{y}f}\varphi^2 B^{-\f43} W^2 dy\\
        =:&\sum_{j=1}^4 T_{\beta;4j}.\\
        % \leq&\frac{3\al}{4}\phi\ep^{1/2}\nu\|\pa_y \pa_{y}f_\bk\|_{L^2}^2+\frac{\beta^2}{\al}\phi^3 \ep^{-1/2}\nu\|\sin(y)\|_{L^\infty}^2\|f_\bk\|_{L^2}^2\\
        % &+\frac{\beta^2}{2\al\gamma}\gamma\nu^{\f12}|k_m|^{\f12}\phi^3\|\cos(y)\pa_y f_\bk\|_{L^2}^2-\beta\phi^2|k_m|\|\cos(y)f_\bk\|_{L^2}^2\\
        % &\mathhlr{ -\beta\phi^2\operatorname{sign}(k_m)\sum_{j=1}^dk_j \Re\int \cos(y)\sin(y)f_\bk\overline{\pa_{y}f_\bk}\dy}
        \end{split}
    \end{align} We note that the $T_{\beta;43}$ is the coercive term and $T_{\beta;44}$ cancel the last term of $T_{\beta;3}$. 
    Now we estimate the remaining terms in \eqref{T_beta4}. To begin with, we consider the $T_{\beta;41}$,
    \begin{align}\begin{split}
        |T_{\beta;41}|\leq& \frac{1}{8}\beta^\f12\nu^{\f53}\|\pa_y^2f \varphi^\f12 B^{-\f{1}{3}}W\|_{L^2}^2+4\beta^{\f32}\nu\|U''\|_{L^\infty}^2\|f\varphi^\f32 B^{-1}W\|_{L^2}^2\\
        &+4\beta^{\f32}\nu\|U'\pa_y f \varphi^{\f32}B^{-1}W\|_{L^2}^2\\
        \leq& \frac{1}{8}\beta^\f12\nu^{\f53}\|\pa_y^2f \varphi^\f12 B^{-\f{1}{3}}W\|_{L^2}^2+4\beta^{\f32}\nu^{\f{1}{3}}\|U''\|_{L^\infty}^2\|f\varphi B^{\f{1}{3}}W\|_{L^2}^2\\
        &+4\beta^{\f32}\nu\|U'\pa_y f \varphi^{\f32}B^{-1}W\|_{L^2}^2.
    \end{split}
    \end{align}
    Next, we estimate the $T_{\beta;42}$-term. We apply the definitions of $B\geq\nu^{\f14}$ \eqref{B}, $\varphi\leq 1$ \eqref{varphi} and the $W$-estimate \eqref{W_est_y} to obtain that 
    \begin{align}\begin{split}
        &|T_{\beta;42}|\leq \beta\nu^\f43\int |U' f||\pa_y^2 f|\Bigl[\Bigl|\f43\underbrace{\nu^\f{1}{3}\frac{B'}{B^{\f{1}{3}}}\min\{t,\nu^{-\f12}\}}_{\varphi B'B^{-1}}\Bigr|\varphi B^{-\f43}W^2+\frac{4}{3}\frac{|B'|}{B^{\f{7}{3}}}\varphi^2 W^2\\
        &\hspace{4.25cm}+\frac{4}{3}\varsigma\nu^\f{1}{3}\frac{\|U''\|_{L^\infty}}{B^{\f{1}{3}}}\min\{t,\nu^{-\f12}\}\varphi^2 B^{-\f43}W^2\Bigr]dy\\ &
        \leq\frac{\beta^{\f12}}{24}\nu^{\f53}\|\pa_y^2f \varphi^\f12 B^{-\f{1}{3}}W\|_{L^2}^2+32\beta^{\f32}\nu\|U''\|_{L^\infty}^2\|U' f \varphi^{\f32}B^{-2}W\|_{L^2}^2+32\beta^{\f32}\nu\|U' f \varphi^{\f32}B^{-2}W\|_{L^2}^2\\
        &\quad+32\varsigma^2\beta^{\f32}\nu^\f53\min\{t,\nu^{-\f12}\}^2\|U''\|_{L^\infty}^2\|U' f \varphi^{\f32}B^{-\f43}W\|_{L^2}^2\\
         &\leq\frac{\beta^{\f12}}{24}\nu^{\f53}\|\pa_y^2f \varphi^\f12 B^{-\f{1}{3}}W\|_{L^2}^2+32\beta^{\f32}\|U''\|_{L^\infty}^2\nu^{\f{1}{3}}\|U' f \varphi B^{-\f{2}{3}}W\|_{L^2}^2+32\beta^{\f32}\nu^\f{1}{3}\|U' f \varphi B^{-\f23}W\|_{L^2}^2\\
        &\quad+32\varsigma^2\beta^{\f32}\nu^\f{1}{3}\min\{\nu t^2,1\}\|U''\|_{L^\infty}^2\|U' f \varphi B^{-\f23}W\|_{L^2}^2.
    \end{split}
    \end{align}%22 22 2 2
    In the above estimation, we constantly use the relation $B\geq\nu^{\f14}$ to adjust the weights. For example, we have the relation $B^{-4}\leq \nu^{-\f23}B^{-\f43}$. Moreover, we have the relation $\nu\min\{t,\nu^{-\f12}\}^2\leq \min\{1,\nu t^2\}. $
   Hence, 
    \begin{align}\begin{split}
        T_\beta \leq& -\beta\lf(\f12-32\sqrt{\beta}-64\sqrt{\beta}\|U''\|_{L^\infty}^2\rg)\nu^\f{1}{3}\|U' f \varphi B^{-\f23}W\|_{L^2}^2+  8\beta \nu \|\pa_{y}f W\|_{L^2}^2\\ &+\frac{5}{12}\beta^{\f12}\nu^{\f53} \|\pa_{y}^2f\sqrt{\varphi} B^{-\f{1}{3}}W\|_2^2 +  5\beta^\f32\nu\|U'\pa_y f\varphi^\f32B^{-1}W\|_2^2\\
        &+12\beta^{\f32}\nu^{\f{1}{3}}\|U''\|_{L^\infty}^2\|f\varphi B^{\f{1}{3}}W\|_{L^2}^2.
    \end{split}
    \end{align}
    This concludes the proof.
\end{proof}
\begin{proof}[Proof of Lemma~\ref{lem:nondegenerate gamma}]
    Through substitution of the equation~\eqref{EQ:hypo_PS} and integration by parts, we produce
    \begin{align}\begin{split}
        T_\gamma \leq & 12\beta^{\f32}\int |U'|^2|f|^2 \nu^\f{1}{3}B^{\f23}\varphi^2 B^{-2} W^2 dy +8\beta^{\f32}\int |U'|^2|f|^2 \varsigma\nu^\f{1}{3}B^{\f23} \varphi^3 B^{-2} W^2 dy\\ &+8\beta^{\f32}\Re\int |U'|^2\pa_t f\overline{f} \varphi^3 B^{-2}W^2 dy\\
        \leq& (12+8\varsigma)\beta^\f32\nu^\f{1}{3}\|U'f\varphi B^{-\f23}W\|_2^2\\
        &+ 8\beta^\f32\lf( \Re\int |U'|^2\lf(\nu\pa^2_{y}f - iUf\rg)\overline{f}\varphi^3 B^{-2}W^2dy\rg)\\
        \leq& 20\beta^\f32\nu^\f{1}{3}\|U'f\varphi B^{-\f23}W\|_2^2 -8\beta^\f32\nu\|U'\pa_{y}f\varphi^\f32 B^{-1}W\|_{L^2}^2\\
        &- 8\beta^\f32\nu\Re\int 2U''U'\pa_{y}f\ \overline{ f}\varphi^3 B^{-2}W^2dy- 8\beta^\f32\nu\Re\int |U'|^2\pa_{y}f\ \overline{ f}\pa_y(\varphi^3 B^{-2}W^2)dy\\
        =:&\sum_{j=1}^4T_{\gamma;j}.\end{split}
        \end{align}
        We estimate the $T_{\gamma;3}$-term as follows
        \begin{align}
            |T_{\gamma;3}|\leq& 8\beta^\f32 \nu^{\f{1}{3}}\|U' f\varphi B^{-\f23}W\|_{L^2}^2+8\beta^\f32\nu^{\frac{5}{3}}\|U''\|_{L^\infty}^2\|\pa_y fW\|_{L^2}^2\|B^{-4/3}\|_{L^\infty}^2\\
\leq&            8\beta^\f32 \nu^{\f{1}{3}}\|U' f\varphi B^{-\f23}W\|_{L^2}^2+8\beta^\f32\nu\|U''\|_{L^\infty}^2\|\pa_y fW\|_{L^2}^2. 
        \end{align}
Finally, we estimate the $T_{\gamma;4}$-term,
\begin{align}\begin{split}
    &|T_{\gamma;4}|\\
     &\leq8\beta^{\f32}\nu \int|U'|^2|\pa_yf||f|\Big|2\nu^{\f{1}{3}}\frac{B'}{B^{\f{1}{3}}}\min\{t,\nu^{-\f12}\} \varphi^2 B^{-2}W^2-2 \frac{B'}{B^{3}}\varphi^3W^2+\frac{4}{3}\varsigma\nu^{\f{1}{3}}B'B^{-\f{1}{3}}\min\{t,\nu^{-\f12}\}\varphi^3B^{-2} W^2\Big|dy\\
    &\leq 48\beta\nu\|\pa_y f W\|_{L^2}^2+ 8\|U''\|_{L^\infty}^2\beta^2\nu\|U' f\varphi B^{-\f23}W\|_{L^2}^2\|B^{-\f{8}3}\|_{L^\infty}\\
    &\quad +\f{16}{9}\varsigma^2\beta^{2}\min\{\nu^{-1}, t^2\}\nu^\f53\|U''\|_{L^\infty}^2\|U' f \varphi B^{-\f23}W\|_{L^2}^2\|B^{-\f43}\|_{L^\infty}\\
    &\leq 48\beta\nu\|\pa_y f W\|_{L^2}^2+ \lf(8+\f{16}{9}\varsigma^2\min\{1, \nu t^2\}\rg)\|U''\|_{L^\infty}^2\beta^2\nu^\f{1}{3}\|U' f\varphi B^{-\f23}W\|_{L^2}^2.
\end{split}
\end{align}
        Hence,
        \begin{align}
        T_{\gamma}
        \leq& -8\beta^\f32\nu\|U'\pa_{y}f\varphi^\f32 B^{-1}W\|_{L^2}^2+(48+8\|U''\|_{L^\infty}^2)\beta\nu\|\pa_y f W\|_{L^2}^2\\
        &+ \lf[28+\lf(8+\f{16}{9}\varsigma^2\min\{1, \nu t^2\}\rg)\|U''\|_{L^\infty}^2\rg]\beta^\f32\nu^\f{1}{3}\|U' f\varphi B^{-\f23}W\|_{L^2}^2.
     %\begin{split}   & - 2\gamma\phi^3\nu^{\f12}|k_m|^{1/2}\|\cos(y) \pa_y f_\bk\|_2^2\end{split}
    \end{align}This concludes the proof of \eqref{nondegenerate gamma estimate}.
\end{proof}

\appendix
\section{Technical Lemmas}
\begin{lemma}\label{lem:spec}
Consider shear flow profiles $U$ with finitely many non-degenerate critical points. For $\nu\in (0, 1]$, there exists a constant $\mathfrak{C}(U)$ such that the following estimate holds
    \begin{align}\label{spectral}
{\nu}^{\f{1}{3}}\|f B^{\f{1}{3}}\varphi W\|_{L^2(\Torus)}^2\leq \nu\|\pa_{y} f W\|_{L^2(\Torus)}^2+\mathfrak{C}\nu^\f{1}{3}\lf\| U' B^{-\f23}f \varphi W\rg\|_{L^2(\Torus)}^2. 
    \end{align} 
\end{lemma}
%\begin{remark}
%    This estimate is a consequence of the uncertainty principle in $\rr^d. $ The key in the estimate is that the bound only depends on the regularity in the $y_j$-direction. 
%\end{remark}
\begin{proof} Assume that $U$ has $M$ critical points $\{y_j\}_{j=1}^M$. We choose a parameter $\vartheta(U)\in(0,1)$ such that \begin{align}
    &\text{dist}\Bigl(B(y_j;\vartheta \nu^{\f14}),B(y_{j'};\vartheta \nu^{\f14} )\Bigr)\geq \frac12\min_{j\neq j'}\{\operatorname{dist}(y_{j},y_{j'})\},\quad \forall j\neq j';\\
    &|U'(y)|\gtrsim_U \nu^{\frac{1}{4}},\quad\forall y\notin \bigcup_{j=1}^M B\lf(y_j;\frac{\vartheta}{2} \nu^{\f14}\rg);\quad 
    |U'(y)|\approx_U\operatorname{dist}(y_{j},y),\quad\forall y\in  B\lf(y_j;{\vartheta} \nu^{\f14}\rg). \end{align}
As a consequence of the last relation and the definition of $B(y)=\max\{|U'(y)|,\nu^{\f14}\}$, we have that if $\vartheta(U)$ is chosen small enough, for all $y\in  \bigcup_{j=1}^M B\lf(y_j;{\vartheta} \nu^{\f14}\rg)$, 
\begin{align}
    |U'(y)|\leq C(U)\vartheta \nu^{\f14}\leq \frac{1}{2}\nu^{\f14}\;\Longrightarrow\; B(y)\equiv\nu^{\f14}.
\end{align}
Hence,
    \begin{align}
   & \pa_y W(t,y)\equiv 0,\quad\forall y\in \bigcup_{j=1}^M B\lf(y_j;{\vartheta} \nu^{\f14}\rg).
\end{align}
Now we define a smooth partition of unity $\{\chi_i\}_{i=0}^M$,
\begin{align}
    &\chi_j(y)=\begin{cases}1,&\quad \operatorname{dist}(y,y_j)\leq \frac{1}{2}\vartheta\nu^\f14,\\
    0,&\quad \operatorname{dist}(y,y_j)\geq \vartheta\nu^\f14,\end{cases}\quad j\neq 0;\quad\\
    & \chi_0(y)=1-\sum_{j=1}^M\chi_j(y);\quad \max_{j=0, \cdots, M}\|\chi'_j\|_{L^\infty}\leq C(U)\nu^{-\f14}. 
\end{align} We decompose the function $f(y)=f(y)(\sum_{j=0}^M \chi_j)=\sum_{j=0}^M \pw fj(y)$. Now for the $L^2$-estimate of $\pw fj,\, j\neq 0$, we use the integration by parts formula to bound. For example, we consider the case where $j=1$, 
\begin{align}
    \nu^{\frac{1}{3}}&\int_{\mathbb{T}} B^{\f{2}{3}}|\pw f1|^2\varphi ^2 W^2 dy\\
    \leq & C(U)\nu^{\frac{1}{2}}\bigg|\int_{y\in \Torus} |\pw f1|^2 \pa^2_{y}U \varphi^2 W^2 dy\bigg| = C(U)\nu^{\f12}\bigg|\int_{y\in \Torus} \pa_{y}|\pw f1|^2 U'\varphi ^2 W^2 dy\bigg|\\
    \leq& \frac{1}{2}\nu\|\pa_{y} \pw f1W\|_{L^2(\operatorname{supp}\chi_1)}^2+ C(U)\lf\|U' \pw f1 \varphi W\rg\|_{L^2}^2\\
    \leq & \frac{1}{2}\nu \|\pa_y f \varphi W\|_{L^2(\operatorname{supp}\chi_1)}^2+C(U)\nu\lf\|\f{U'}{\nu^\f14}\frac{|U'|^{\f23}\vee\nu^\f{1}{6}}{B^{\f23}}f\varphi W\rg\|_{L^2(\operatorname{supp}\chi_1')}^2\|\chi_1'\|_{L^\infty}^2+C(U)\lf\|U'\frac{|U'|^{\f23}\vee\nu^\f{1}{6}}{B^{\f23}} f\chi_1 \varphi W\rg\|_{L^2}^2\\
    \leq&\frac{1}{2}\nu \|\pa_y f W\|_{L^2(\operatorname{supp}\chi_1)}^2+C(U)\nu^\f{1}{3}\lf\|U'{B^{-\f23}} f\chi_1 \varphi W\rg\|_{L^2}^2.\label{spec_nr_crit}
\end{align}
Since the supports of the cutoff functions $\{\chi_j\}_{j\neq 0}$ are disjoint, we have that 
\begin{align}
\nu^{\f{1}{3}}\int_{\Torus} B^{\f23}|f(1-\chi_0)|^2 \varphi ^2W^2dy\leq \frac{1}{2}\nu \|\pa_{y}f\|_{L^2}^2+C(U)\nu^{\f{1}{3}}\| fU'B^{-\f23}\varphi W\|_{L^2}^2.
\end{align}
We further observe that, since the $|U'(y)|\geq \frac{\nu^{\f14}}{C(U)}$ on the support of $\chi_0$, 
\begin{align}
    \nu^{\f{1}{3}}\|f\chi_0 B^{\f{1}{3}} \varphi W\|_{L^2}^2\leq& C(U)\nu^{\f{1}{3}}\int |f\chi_0|^2|U'|^{\f23} \varphi ^2 W^2dy= C(U)\nu^{\f{1}{3}}\int |f\chi_0|^2\frac{|U'|^2}{|U'|^{\f43}}\varphi ^2 W^2dy\\
    \leq&C(U)\nu^\f{1}{3} \|fU' B^{-\f23}\chi_0\varphi W\|_{L^2}^2.
\end{align} Here, we have used the fact that on the support of $\chi_0$, $B=\max\{\nu^\f14,|U'|\}\leq C(U) |U'|. $ So $\frac{1}{|U'|^\f43}\mathbf{1}_{\operatorname{supp}\chi_0}\leq \frac{C(U)}{B^{\f43}}\mathbf{1}_{\operatorname{supp}\chi_0}$.  
Combining the above estimates, we have that 
\begin{align}
      \nu^{\frac{1}{3}}\| B^{\f{1}{3}} f \varphi W\|_{L^2}^2\leq& \nu \|\pa_y f W\|_{L^2}^2+C(U)\nu^\f{1}{3}\lf\|U'{B^{-\f23}} f\varphi W\rg\|_{L^2}^2.
\end{align}This concludes the proof of the lemma.
\end{proof}
\begin{lemma}\label{lem:W_est}
    Consider the Lipschitz function $W$ \eqref{W}. Further assume that the parameters $\nu,\varsigma\in(0, 1]$. The following estimates hold
    \begin{align}\label{W_est_t}
        |\pa_t W(t,y)|\leq&\varsigma\nu^{\f{1}{3}}B^{\f23}\mathbbm{1}_{t\in[\nu^{-\f{1}{3}}B^{-\f23}, \nu^{-\f12}]}W(t,y),\quad \text{a.e.},\\
\label{W_est_y}        |\pa_y W(t,y)|\leq&\frac{2\varsigma}{3}\mathbbm{1}_{t\geq\nu^{-\f{1}{3}}B^{-\f23}}\nu^{\f{1}{3}}B^{-\f{1}{3}}\min\{t,\nu^{-\f12}\}\|U''\|_{L^\infty} W(t,y) ,\quad \text{a.e.}
    \end{align}%\quad \varsigma\in(0,1) 
\end{lemma}
\begin{proof}First of all, we recall the exponential part of the temporal weight $W$ \eqref{W} and rephrase it in terms of absolute value functions
    \begin{align}
\log W&=\varsigma \nu^\f{1}{3}B^{\f23}\max\bigl\{\nu^{-\f{1}{3}}B^{-\f23},\min\{t,\nu^{-\f12}\}\bigr\}\\
&=\begin{cases}\varsigma \nu^\f{1}{3}B^{\f23}\max\bigl\{\nu^{-\f{1}{3}}B^{-\f23}, t\bigr\},&\quad t\leq \nu^{-\f12}\\
\varsigma \nu^\f{1}{3}B^{\f23}\max\bigl\{\nu^{-\f{1}{3}}B^{-\f23},\nu^{-\f12}\bigr\},&\quad t\geq \nu^{-\f12}
\end{cases}\\
&=\begin{cases}\varsigma \max\bigl\{1, \nu^\f{1}{3}B^{\f23}t\bigr\},&\quad t\leq \nu^{-\f12}\\
\varsigma \nu^{-\f{1}{6}}B^{\f23},&\quad t\geq \nu^{-\f12}
\end{cases}\\
&=\begin{cases}\displaystyle \f{{\varsigma}}{2}\bigl[1+ \nu^\f{1}{3}B^{\f23}t+ {|1-\nu^\f{1}{3}B^{\f23}t|} \bigr],&\quad t\leq \nu^{-\f12}\\
\varsigma \nu^{-\f{1}{6}}B^{\f23},&\quad t\geq \nu^{-\f12}
\end{cases}. 
\end{align}
%7
Here, in the computation, we have used the assumption that $\nu\leq 1$ and $B=\max\{|U'|,\nu^{\f14}\}\geq \nu^{\f14}$ \eqref{B}. As a consequence, it is natural to distinguish between the $t\leq \nu^{-\f12}$ and the $t> \nu^{-\f12}$ cases. 

In the case where $t\leq \nu^{-\f12}$, we take the time derivative to obtain that
\begin{align}
    |\pa_t W|=\f\varsigma 2 \lf|\nu^\f{1}{3}B^{\f23}\lf(1-\frac{1-\nu^{\f{1}{3}}B^{\f23}t}{|1-\nu^{\f{1}{3}}B^{\f23}t|}\rg)\rg|W\leq \mathbbm{1}_{t\in[\nu^{-\f{1}{3}}B^{-\f23}, \nu^{-\f12}]}\varsigma\nu^{\f{1}{3}}B^\f23 W,\quad \text{a.e.}
\end{align}
Next, we estimate the spatial derivative. We first estimate $|B'|$. Since $B(y)=\max\{|U'(y)|,\nu^{1/4}\}$ is Lipschitz, $B'$ exists a.e. and one has the relation 
\[
B'(y)=\frac{U'(y)}{|U'(y)|}\,U''(y)\,\mathbbm{1}_{\{|U'(y)|\ge \nu^{1/4}\}}
\quad\text{a.e.}
\]
Hence $|B'(y)|\le |U''(y)|\le \|U''\|_{L^\infty}$ a.e.
% \begin{align}
% \|B'\|_{L^\infty}=\lf\|\frac{U'}{|U'|}U''\mathbbm{1}_{|U'|\geq \nu^{\f14}}(\cdot)\rg\|_{L_y^\infty}\leq \|U''\|_{L^\infty}\;\Longrightarrow\; |B'(y)|\leq \|U''\|_{L^\infty} \text{ a.e.}.
% \end{align}
Now, taking the spatial derivative of $W$ yields that 
\begin{align}
    |\pa_y W(t,y)|=&\frac{\varsigma}{2}\lf|\frac{2}{3}\nu^{\f{1}{3}}B^{-\f{1}{3}}B't\lf(1-\f{1-\nu^{\f{1}{3}}B^{\f23}t}{|1-\nu^{\f{1}{3}}B^{\f23}t|}\rg)\rg|W(t,y)\\
    \leq& \frac{2\varsigma}{3}\mathbbm{1}_{t\in[\nu^{-\f{1}{3}}B^{-\f23}, \nu^{-\f12}]}\nu^{\f{1}{3}}B^{-\f{1}{3}}t\|U''\|_{L^\infty} W(t,y),\quad \text{a.e.}
\end{align}

Finally, for the $t>\nu^{-\f12}$ case, we have that 
\begin{align}
    |\pa_t W(t,y)|\equiv 0.
\end{align}
The spatial derivative can be estimated as follows
\begin{align}
    |\pa_y W(t,y)|= \lf|\frac{2\varsigma}{3}\nu^{-\f{1}{6}}B^{-\f{1}{3}}B'\rg| W\leq \frac{2\varsigma}{3}\nu^{-\f{1}{6}}B^{-\f{1}{3}}\|U''\|_{L^\infty_y} W(t,y),\quad \text{a.e.}
\end{align}
Combining the estimates above yields the estimates \eqref{W_est_t} and \eqref{W_est_y}. 
\end{proof}
    \bibliography{Sim}
\end{document}